\documentclass[10pt]{amsart}
\usepackage{mathrsfs,amssymb}
\usepackage{graphicx}
\begin{document}

\newtheorem{theorem}{Theorem}
\newtheorem{proposition}[theorem]{Proposition}
\newtheorem{lemma}[theorem]{Lemma}
\newtheorem{corollary}[theorem]{Corollary}
\newtheorem{conjecture}[theorem]{Conjecture}
\newtheorem{question}[theorem]{Question}
\newtheorem{problem}[theorem]{Problem}
\theoremstyle{definition}
\newtheorem{definition}{Definition}

\theoremstyle{remark}
\newtheorem{remark}[theorem]{Remark}

\renewcommand{\labelenumi}{(\roman{enumi})}
\def\theenumi{\roman{enumi}}


\renewcommand{\Re}{\operatorname{Re}}
\renewcommand{\Im}{\operatorname{Im}}

\def\scrA{{\mathcal A}}
\def\scrB{{\mathcal B}}
\def\scrD{{\mathcal D}}
\def\scrL{{\mathcal L}}
\def\scrS{{\mathcal S}}

\def \G {{\Gamma}}
\def \g {{\gamma}}
\def \R {{\mathbb R}}
\def \C {{\mathbb C}}
\def \Z {{\mathbb Z}}
\def \Q {{\mathbb Q}}
\def \TT {{\mathbb T}}
\newcommand{\T}{\mathbb T}
\def \GinfmodG {{\Gamma_{\!\!\infty}\!\!\setminus\!\Gamma}}
\def \GmodH {{\Gamma\setminus\H}}
\def \vol {\hbox{vol}}
\def \sl  {\hbox{SL}_2(\mathbb Z)}
\def \slr  {\hbox{SL}_2(\mathbb R)}
\def \psl  {\hbox{PSL}_2(\mathbb R)}

\newcommand{\mattwo}[4]
{\left(\begin{array}{cc}
                        #1  & #2   \\
                        #3 &  #4
                          \end{array}\right) }

\newcommand{\rum}[1] {\textup{L}^2\left( #1\right)}
\newcommand{\norm}[1]{\left\lVert #1 \right\rVert}
\newcommand{\abs}[1]{\left\lvert #1 \right\rvert}
\newcommand{\inprod}[2]{\left \langle #1,#2 \right\rangle}

\renewcommand{\i}{{\mathrm{i}}}

\newcommand{\area}{\operatorname{area}}
\newcommand{\ecc}{\operatorname{ecc}}

\newcommand{\Op}{\operatorname{Op}}
\newcommand{\dom}{\operatorname{Dom}}
\newcommand{\Dom}{\operatorname{Dom}}
\newcommand{\tr}{\operatorname{tr}}

\newcommand{\Norm}{\mathcal N}
\newcommand{\simgeq}{\gtrsim}%
\newcommand{\simleq}{\lesssim}

\newcommand{\UN}{U_N}
\newcommand{\OPN}{\operatorname{Op}_N}
\newcommand{\HN}{\mathcal H_N}
\newcommand{\TN}{T_N}  
\newcommand{\PDO}{\Psi\mbox{DO}}

\newcommand{\intinf}{\int_{-\infty}^\infty}
\newcommand{\lcm}{\operatorname{lcm}}
\newcommand{\E}{\mathbb E}
\newcommand{\Prob}{\operatorname{Prob}}
\newcommand{\Var}{\operatorname{Var}}
\newcommand{\conv}{*}

\title{A metric theory of minimal gaps}
\author{Ze\'ev Rudnick}
\address{Raymond and Beverly Sackler School of Mathematical Sciences,
Tel Aviv University, Tel Aviv 69978, Israel}
\email{rudnick@post.tau.ac.il}

\date{\today}

\begin{abstract}
We study the minimal gap statistic for fractional parts of sequences of the form $\mathcal A^\alpha = \{\alpha a(n)\}$ where $\mathcal A = \{a(n)\}$ is a sequence of distinct of integers. Assuming that the additive energy of the sequence is close to its minimal possible value, we show that for almost all $\alpha$, the minimal gap $\delta_{\min}^\alpha(N)=\min\{\alpha a(m)-\alpha a(n)\bmod 1: 1\leq m\neq n\leq N\}$ is close to that of a random sequence. \end{abstract}

\maketitle

We start with a sequence of points $\mathcal X=\{x_n:n=1,2,\dots\}\subset \R/\Z$ in the unit interval/circle, which we assume is asymptotically uniformly distributed: 
For any subinterval $\subset \R/\Z$, we have 
\begin{equation}\label{UD}
\lim_{N\to \infty} \frac 1N\#\{n\leq N: x_n\in I\} = |I| \;.
\end{equation}
In particular, the mean spacing between the points lying in any subinterval is $1/N$. Our goal is to understand the {\em minimal gap} 
$$ \delta_{\min}(\mathcal X,N) = \min(|x_n-x_m|: n,m\leq N, n\neq m)$$
(with  a suitable  modification for wrapping around). 

For {\em random} points,  namely  $N$ independent uniform points in the unit interval (Poisson process), the minimal gap is almost surely of size $1/N^2$ \cite{Levy}. In this note we study the metric theory of the minimal gap statistic for a class of deterministic sequences of fractional parts, such as fractional parts of polynomials. The case of quadratic polynomials $x_n =\alpha n^2$ has its roots in the recent paper \cite{BBRR}, which studies the more complicated case of the minimal gap statistic for the sequence of eigenvalues of the Laplacian on a rectangular billiard, namely the points 
$\{\alpha m^2+n^2 : m,n\geq 1\}$ on the real line. 

We fix   a sequence   $\mathcal A=\{a(n):n=1,2,\dots\}\subset \mathbb Z$ of distinct integers ($a(n)\neq a(m)$ if $m\neq m$), and study the minimal gap statistic of fractional parts of the set 
$$\mathcal A^\alpha = \{\alpha a(n) \bmod 1:n=1,2, \dots\}\subset \R/\Z \;.
$$
(it is an old result of Weyl that $\mathcal A^\alpha$ satisfies \eqref{UD} for almost all $\alpha$). 
We want to know under which conditions we can show that for {\em almost all} $\alpha$, the minimal gap statistics 
$$\delta_{\min}^\alpha(N)=\delta_{\min}( \mathcal A^\alpha,N)$$ 
follows that of the random case, that is of size about $1/N^2$ for almost all $\alpha$.  
It is easy to see that we cannot have much smaller minimal gaps:
\begin{theorem}\label{thm delta big}
Assume that $\mathcal A$ consists of distinct integers.  Then for all  $\eta>0$, for almost all $\alpha$, 
$$
\delta_{\min}^\alpha(N)>\frac 1{N^{2+\eta}},\quad \forall N>N_0(\alpha) \;.
$$
\end{theorem}

To make the minimal gap small, we give a criterion in terms of the ``additive energy" $E(\mathcal A,N) $ of the sequence: 
$$
E(\mathcal A,N) :=\#\{(n_1,n_2,n_3,n_4)\in [1,N]^4: a(n_1)+a(n_2) = a(n_3)+a(n_4) \}  \;.
$$ 
Note that $N^2\leq E(\mathcal A,N) \leq N^3$. 
The result is
\begin{theorem}\label{thm min}
Assume that $\mathcal A$ consists of distinct integers, and that the additive energy satisfies
$$ E(\mathcal A,N)\ll N^{2+o(1)},\quad \forall N\gg 1\;.
$$
Then  for all $\eta>0$, for almost all $\alpha$, 
$$
      \delta_{\min}^\alpha(N)<\frac 1{N^{2-\eta}},\quad \forall N>N_0(\alpha) \;.
$$
\end{theorem}
 Examples:  
For $a(n)=n^d$, $d\geq 2$, it is shown in \cite{RS} that the $E(\mathcal A,N)\ll N^{2+o(1)}$. For lacunary sequences, we have   $E(\mathcal A,N)\ll N^2$  \cite{RZ}.  Hence Theorem~\ref{thm min} applies to these sequences.

Relaxing the required bound on the additive energy will give a weaker result on the minimal spacing, basically that $\delta_{\min}^{\alpha}(N)<E(\mathcal A,N)/N^{4-\eta}$ almost surely. For this to be nontrivial, we need the additive energy to be no bigger than $E(\mathcal A,N)\ll N^{3-\eta'}$ for some $\eta'>0$. A notable case where the additive energy is bigger is that of $\mathcal A=\mathcal P$ being the sequence of primes, where $E(\mathcal P,N) \approx N^3/\log N$.  
In this case, we cannot have gaps much larger than the average gap: A simple argument shows that given any $\varepsilon>0$, for almost all $\alpha$, we have $\delta_{\min}(\mathcal P^\alpha,N)\gg 1/(N(\log N)^{2+\varepsilon})$, see \S~\ref{sec:dense}.



\noindent{\bf Acknowledgements:} We thank Andrew Granville and Niclas Technau for their comments. 
The work was supported  by the European Research Council under the European Union's Seventh
Framework Programme (FP7/2007-2013)/ERC grant agreement n$^{\text{o}}$ 320755.   
 
\section{A bilinear statistic }

To study the minimal gap, we introduce a statistics counting all possible gaps: 
We start with a smooth, compactly supported window function $f\in C_c^\infty([-1/2,1/2])$, which is non-negative: $f\geq 0$, 
and of unit mass $\int f(x)dx=1$, and define 
$$F_M(x) = \sum_{j\in \Z} f(M(x+j))$$
which is localized on the scale of $1/M$, and periodic: $F_M(x+1) = F_M(x)$. We then set 
$$
D_{\mathcal A}(N,M)(\alpha) = \sum_{\substack{1\leq m,n\leq N\\m\neq n}} F_M(\alpha a(n)-\alpha a(m)) \;.
$$

The expected value of $D_{\mathcal A}(N,M)$ is easily seen to equal 
\begin{equation}\label{E(D)} 
\int_0^1 D_{\mathcal A}(N,M)(\alpha)d\alpha = \frac{N(N-1)}{M}\sim \frac{N^2}{M}\;.
\end{equation}
This already suffices to show that minimal gaps cannot typically be small (Theorem~\ref{thm delta big}), see \S~\ref{sec:lower}.

We will bound the  variance of $D_{\mathcal A}(N,M)$, from which Theorem~\ref{thm min} will follow: 
\begin{proposition}\label{prop var}
\begin{equation*}
\Var D_{\mathcal A}(N,M) \ll \frac{1}{M} N^{\epsilon}E(\mathcal A,N)  \;.
\end{equation*}
\end{proposition}

The statistic $D_{\mathcal A}(N,M)(\alpha)$ is related to the pair correlation function of the sequence $\mathcal A^\alpha$, which in our notation is $D_{\mathcal A}(N,N)(\alpha)/N$. Pair correlation measures gaps on the scale of the mean spacing, assumed here to be $1/N$, corresponding to $M=N$, here we are looking at much smaller scales of $M$ close to $N^2$. 

The metric theory of the pair correlation function of fractional parts was initiated in \cite{RS}, where the sequences $a(n) = \alpha n^d$ were shown to almost surely have Poissonian pair correlation for $d\geq 2$ (see \cite{MS, HB} for different proofs of the quadratic case $d=2$). The problem has since been studied in several other cases  and  has recently been revived   in an abstract setting \cite{ALL, Walker,  LT, BCGW}. In particular, a convenient criterion for almost sure Poisson pair correlation has been formulated by Aistleitner,  Larcher and Lewko \cite{ALL} in terms of the additive energy $E(\mathcal A,N)$ of the sequence. 
The proof of Proposition~\ref{prop var} is close to that of the analogous statement for the pair correlation function in \cite{ALL}, which in turn is based on \cite{RS, RZ}. 

\begin{proof}


The Fourier expansion of $F_M(x)$ is 
\begin{equation}
F_M(x) = \sum_{k\in \Z} \frac 1M \widehat{f}(\frac kM) e(kx)
\end{equation}
where $\widehat{f}(y) = \int_{-\infty}^\infty f(x)e^{-2\pi i xy }dx$. Inserting into the definition of $D(N,M)$ gives 
\begin{equation}
D(N,M)(\alpha) =\sum_{k\in \Z}  \frac 1M \widehat{f}(\frac kM)  \sum_{\substack{1\leq m,n\leq N\\m\neq n}}  e(k\alpha(a(m)-a(n)) \;.
\end{equation}
Integrating over $\alpha$ gives the expected value (we assume $a(m)\neq a(n)$ if $n\neq m$) 
$$
\int_0^1 D(N,M)(\alpha)d\alpha =\widehat{f}(0)\frac 1M N(N-1) =\frac{N(N-1)}{M} \;.
$$

The variance is the second moment of the sum over nonzero frequencies:
$$
\Var D(N,M) = \int_0^1 \Big| \sum_{0\neq k\in Z}  \frac 1M \widehat{f}(\frac kM)  \sum_{\substack{1\leq m,n\leq N\\m\neq n}}  e(k\alpha(a(m)-a(n)) \Big|^2 d\alpha \;.
$$
Squaring out and integrating gives 
\begin{multline*}
\Var D(N,M) = 
 \sum_{k_1,k_2\neq 0} \frac 1{M^2}  \widehat{f}(\frac {k_1}M) \widehat{f}(\frac {k_2}M) \times 
\\
 \#\{m_1\neq n_1, m_2\neq n_2: k_1(a(m_1)-a(n_1)) = k_2 a(m_2)-a(n_2))\} \;.
\end{multline*}
We now follow \cite[Lemma 3]{ALL} to convert this to ``GCD sums". Let 
$$
R(v) = \#\{1\leq m\neq n\leq N: a(m)-a(n) = v\} \;.
$$ 
Then 
$$
\Var D(N,M) = \sum_{ v_1,v_2\neq 0} \sum_{k_1,k_2\neq 0} c(k_1)c(k_2)  R(v_1)R(v_2)\delta(k_1v_1=k_2v_2)
$$
where we set 
$$ c(k) :=\frac 1M \widehat{f}(\frac kM) \;.
$$

\begin{lemma}
Let $f\in C_c^\infty(\R)$. For any nonzero integers $v_1, v_2\neq 0$, 
\begin{equation}\label{improved gcd}
\sum_{k_1,k_2\neq 0} c(k_1)c(k_2)   \delta(k_1v_1=k_2v_2)
\ll_f   \frac 1M \frac{\gcd(v_1,v_2)}{\sqrt{|v_1v_2|}} 
\;.
\end{equation}
\end{lemma}
\begin{proof}
For $k_1,k_2\neq 0$, 
we have $k_1v_1=k_2 v_2$ if and only if 
$$(k_1,k_2) = \ell (\frac{v_2}{\gcd(v_1,v_2)}, \frac{v_1}{\gcd(v_1,v_2)})
$$
for some nonzero integer $0\neq \ell \in \Z$. Abbreviating $a_i  = v_i/(M\gcd(v_1,v_2))$, we find  
\begin{equation*}
\begin{split}
\sum_{k_1,k_2\neq 0} c(k_1)c(k_2)\delta(k_1v_1=k_2v_2) 
&= \sum_{0\neq \ell\in \Z} c(\ell \frac{v_2}{\gcd(v_1,v_2)} ) c(\ell \frac{v_1}{\gcd(v_1,v_2)}) 
 \\& 
= \frac 1{M^2}  \sum_{0\neq \ell\in \Z}\widehat{f}(a_2\ell ) \widehat{f}(a_1 \ell ) 
\end{split}
\end{equation*}
so that it suffices to show that 
\begin{equation}\label{improved gcd2}
  \sum_{0\neq \ell\in \Z}\widehat{f}(a_1 \ell  ) \widehat{f}(a_2 \ell  ) \ll_f \frac 1{\sqrt{|a_1a_2|}} = \frac{M\gcd(v_1,v_2)}{\sqrt{|v_1v_2|}} \;.
\end{equation}

Applying Cauchy-Schwarz   we get 
$$
\sum_{0\neq \ell\in \Z}\widehat{f}(\ell a_1) \widehat{f}(\ell a_2 ) \leq \Big(\sum_{\ell\neq 0} \widehat{f}(a_1\ell)^2\Big)^{1/2} \Big(\sum_{\ell \neq 0} \widehat{f}(a_2\ell)^2\Big)^{1/2} \;.
$$
We will obtain \eqref{improved gcd2} if we show that for any $a>0$, 
$$
\sum_{\ell \neq 0} \widehat{f}(a\ell)^2 \ll_f \frac 1 a  \;.
$$
Indeed, if $0<a\ll 1$ then we get a Riemann sum for $(\widehat{f}\;)^2$:
$$
\sum_{\ell \neq 0} \widehat{f}(a\ell)^2\sim \frac 1a \intinf \widehat{f}(y)^2 dy =  \frac 1a \intinf f(x)^2 dx \;.
$$
If $ a\gg 1$ then use the decay rate of the Fourier transform: For $y\neq 0$, 
$$|\widehat{f}(y)|\leq \frac 1{2\pi |y|}\intinf|f'(x)|dx \;,
$$
to obtain
$$
\sum_{\ell \neq 0} \widehat{f}(a\ell)^2  \ll  \sum_{\ell\neq 0} (\frac {\intinf |f'(x)|dx}{ |a\ell|})^2 \ll_f \frac 1{a^2}
$$
which for $a\gg 1$ is $\ll 1/a$. 
\end{proof}

   Hence 
$$
\Var D(N,M)\ll     \frac 1M\sum_{ v_1,v_2\neq 0}  R(v_1)R(v_2)\frac{\gcd(v_1,v_2)}{\sqrt{|v_1v_2|}} \;.
$$
According to the GCD bounds of \cite{DH},   
$$
\sum_{ v_1,v_2\neq 0}  R(v_1)R(v_2)\frac{\gcd(v_1,v_2)}{\sqrt{|v_1v_2|}}  \ll   \exp(\frac{10 \log N}{\log\log N})
\sum_v R(v)^2  
$$
(see \cite{BS} for an essentially optimal refinement).
Now 
$$ \sum R(v)^2 = \#\{m_i,n_j\leq N, m_1\neq n_1, m_2\neq m_2 : a(m_1)-a(n_1)=a(m_2)-a(n_2)\}
$$
is at most  the additive energy $E(\mathcal A,N)$. Thus 
$$ \Var D(N,M)\ll \frac{1}{M} N^{\epsilon}E(\mathcal A,N)$$
as claimed. 
\end{proof}

\begin{corollary}\label{almost sure D}
Assume that the additive energy satisfies $E(\mathcal A,N)<N^{2+o(1)}$. If $M<N^{2-\eta}$,  then for almost all $\alpha$, 
$$ D_{\mathcal A}(N,M)(\alpha)\sim \frac{N^2}{M} \;.$$
\end{corollary}

\begin{proof}
Take $N_k = \lfloor k^{4/\eta }\rfloor$, so that $\sum_k N_k^{-\eta/2}<\infty$, 
and pick any $M_k <N_k^{2-\eta}$, then  we find from Proposition~\ref{prop var}
\begin{equation*}
\begin{split}
 \sum_k \int_0^1 \Big| \frac{D(N_k,M_k)(\alpha)}{N_k(N_k-1)/M_k}-1 \Big|^2 d\alpha &= 
\sum_k \frac{\Var D(N_k,M_k)}{ (N_k(N_k-1)/M_k)^2} 
\\
&<\sum_k N_k^{o(1)} \frac{E(N_k)M_k}{N_k^4}\ll
\sum_k \frac 1{N_k^{ \eta/2} }<\infty
\end{split}
\end{equation*}
and so for almost all $\alpha$
\begin{equation}\label{ae for D_k}
D(N_k,M_k)(\alpha) \sim \frac{N_k^2}{M_k}, \quad \forall k>k_0(\alpha)  \;.
\end{equation}
Apriori the set depends on the test function $f$, but that can be taken care of by a standard diagonalization procedure;  for our purposes we only need one test function.

Given $N\gg 1$, there is a unique value of $k$ so that $N_k\leq N<N_{k+1}$. 
Note that $N/N_k = 1+O(N^{-\eta/4})$. 
Since $M<N^{2-\eta}\sim N_k^{2-\eta}<N_{k+1}^{2-\eta}$, we know   \eqref{ae for D_k} that almost surely $D(N_k,M)/(N_k^2/M) \to 1$. 

Note that 
\begin{equation*}\label{lem:monotonicity}
D(N_k, M)\leq D(N,M)  \leq D(N_{k+1},M) \;.
\end{equation*}
This is because the sums $D(N,M)$ consist of non-negative terms, and hence 
\begin{multline*}
D(N,M) = \sum_{1\leq m\neq n\leq N} F_M(\alpha(a(m)-a(n))) 
\\
\geq  \sum_{1\leq m\neq n\leq N_k} F_M(\alpha(a(m)-a(n))) =D(N_k,M)
\end{multline*}
(we dropped all pairs $(m,n) $ where $\max(m,n)>N_k$).

Since  $N/N_k = 1+O(N^{-\eta/4})$, we have  
$$ 
\frac{D(N_k,M)}{N_k^2/M} \leq \frac{D(N,M) }{N^2/M(1+O(N^{-\eta/4}))} = \frac{D(N,M)}{N^2/M}(1+O(N^{-\eta/4}) )
$$
and likewise 
$$ \frac{D(N,M)}{N^2/M}\leq \frac{D(N_{k+1},M)}{N_{k+1}^2/M} (1+O(N^{-\eta/4})) \;.
$$
Since we know that almost surely $D(N_k,M)/(N_k^2/M) \to 1$, we deduce that almost surely also $ D(N,M)/(N^2/M)\to 1$. 
\end{proof}

\begin{corollary}
 Theorem~\ref{thm min} holds.
\end{corollary}
\begin{proof}
Fix $\eta>0$, and let $M=\frac 12 N^{2-\eta}$. Since $D(N,M)(\alpha) \sim \frac{N^2}{M} >1$ by Corollary~\ref{almost sure D}, 
we have a gap of size at most $1/(2M)=1/N^{2-\eta}$, that is $\delta_{\min}^\alpha(N)<1/N^{2-\eta}$ almost surely. 
\end{proof}

\section{Lower bounds: Proof of Theorem~\ref{thm delta big}}\label{sec:lower}
We take any sequence of integers $\mathcal A = \{a(n)\}$ with distinct elements. We want to show that for any $\eta>0$, almost surely, 
$$\delta^\alpha(N)>1/N^{2+\eta}, \quad \forall N>N_0(\alpha)\;.
$$
Let $N_k = \lfloor k\rfloor^{2/\eta}$. We claim that it suffices to show that for almost all $\alpha$, 
\begin{equation}\label{delta small}
\delta_{\min}^\alpha(N_k)>2/N_k^{2+\eta}, \quad \forall k>k_0(\alpha)\;.
\end{equation}

Indeed, note that if $N_k\leq N<N_{k+1}$ then $\delta_{\min}^\alpha(N)\geq \delta_{\min}^\alpha(N_{k+1})$. 
Since $N_{k+1}\sim N$, by \eqref{delta small} we have for almost all $\alpha$  
$$ \delta_{\min}^\alpha(N)\geq \delta_{\min}^\alpha(N_{k+1})>2/N_{k+1}^{2+\eta}>1/N^{2+\eta}$$
for $N>N_0(\alpha)$.

To prove \eqref{delta small}, it suffices,  by the Borel-Cantelli lemma,  to show that 
\begin{equation}\label{sum prob delta}
\sum_k \Prob\Big(\delta_{\min}^\alpha(N_k)\leq 2/N_k^{2+\eta} \Big) <\infty\;.
\end{equation}

In the definition of $D(N,M)=D_f(N,M)$, choose $f$ so that $f(x)\geq 1$ if $|x|\leq 1/4$ (and in addition, $f\geq 0$ is non-negative, $\intinf f(x)dx=1$, $f$ is smooth and supported in $[-1/2,1/2]$). 
Now note that for such $f$, if $D_f(N,M)<1$ then $\delta_{\min}^\alpha(N)>1/(4M)$. This is because $D_f(N,M)$ is a sum of non-negative terms, and if there is one gap of size $\leq 1/(4M)$ then the corresponding term $F_M(\alpha(a(m)-a(n)) = \sum_j f(M(\alpha(a(m)-a(n))+j)\geq 1$ by the choice of $f$, so that $D_f(N,M)\geq 1$. Thus we find that 
$$
\delta_{\min}^\alpha(N)\leq \frac 1{4M} \quad  \Rightarrow \quad D_f(N,M)\geq 1 
$$
and hence 
\begin{equation}\label{bd for prob delta}
\Prob\Big(\delta_{\min}^\alpha(N )\leq \frac 1{4M}\Big) \leq \Prob\Big( D_f(N, M ) \geq 1\Big) \;.
\end{equation}

Now since $D_f\geq 0$, 
$$
\Prob\Big( D_f(N ,M ) \geq  1 \Big) \leq \int_0^1 D_f(N , M)(\alpha)d\alpha
$$
so that by \eqref{E(D)}  for $M_k=\dfrac 18 N_k^{2+\eta}$, 
$$  \int_0^1 D(N_k,M_k)(\alpha)d\alpha \sim 8N_k^{-\eta}\ll \frac 1{k^2}$$
which together with \eqref{bd for prob delta} proves \eqref{sum prob delta}, hence \eqref{delta small}. 
This proves Theorem~\ref{thm delta big}.

\section{Minimal gaps for the primes}\label{sec:dense}

Let $a(n) = p_n$, the $n$-th prime. By Khinchin's theorem, for all $\varepsilon>0$  there is a set of full measure of $\alpha$'s so that 
$|| q\alpha ||\geq 1/(q(\log q)^{1+\varepsilon})$ for any integer $q\geq q_0(\alpha)$.  
In particular, for such $\alpha$, the gap between fractional parts of $\alpha p_n \bmod 1$ are
$$ ||\alpha(p_m-p_n)|| \gg \frac 1{|p_m-p_n|(\log |p_m-p_n|)^{1+\varepsilon}} \geq \frac 1{ N(\log N)^{2+\varepsilon}} \;,
 $$
since $|p_m-p_n|\leq p_N\sim N\log N$ for $m<n\leq N$. Hence  for such $\alpha$, the minimal gap satisfies 
$\delta_{\min}^\alpha(N)>1/N( \log N)^{2+\varepsilon}$. 

A similar argument applies to other dense cases, such as the sequence of squarefree integers.  
An extreme case is that when $\mathcal A=\mathbb N$ is the sequence of all natural numbers. The argument above gives the minimal gap here is, for almost all $\alpha$, at least $\delta_{\min}^\alpha(N)\gg 1/(N(\log N)^{1+\varepsilon})$. 
Note that in this case the ``three-gap" theorem shows that there are at most three distinct gaps between the fractional parts 
$\{\alpha n\bmod 1:n\leq N\}$. 
Concerning other ``dense" sequences, it is known that for any sequence of integers $\mathcal A \subset [1,M]$, the  fractional parts 
$\alpha a(m)\bmod 1$ have at most $O(\sqrt{M})$ distinct gaps \cite{VZ, BGS}.

\end{document}